\documentclass[a4paper,leqno,12pt]{article}

\usepackage{extsizes}
\usepackage{geometry}
\usepackage{sectsty}
\usepackage{setspace}
\usepackage{indentfirst}
\usepackage{amsmath}
\usepackage{amsthm}
\usepackage{amssymb}
\usepackage{hyperref}
\usepackage{caption}
\usepackage{subcaption}
\usepackage{graphicx}
\usepackage{lineno}
\usepackage{cleveref}

\newtheorem{theorem}{Theorem}
\newtheorem{lemma}[theorem]{Lemma}
\newtheorem{proposition}[theorem]{Proposition}

\newtheorem{conjecture}[theorem]{Conjecture}
\newtheorem{corollary}[theorem]{Corollary}
\newtheorem{definition}[theorem]{Definition}

\newtheorem{algorithm}[theorem]{Algorithm}

\newtheorem{remark}[theorem]{Remark}

\numberwithin{equation}{section}
\numberwithin{theorem}{section}

\newcommand{\abs}[1]{|#1|}
\newcommand{\set}[1]{\left\{#1\right\}}

\newcommand{\arrowsv}[0]{\overset{v}{\rightarrow}}
\newcommand{\arrowse}[0]{\overset{e}{\rightarrow}}

\newcommand{\mH}[0]{\mathcal{H}}
\newcommand{\mL}[0]{\mathcal{L}}
\newcommand{\mA}[0]{\mathcal{A}}
\newcommand{\mB}[0]{\mathcal{B}}
\newcommand{\mM}[0]{\mathcal{M}}

\DeclareMathOperator{\V}{V}
\DeclareMathOperator{\E}{E}

\DeclareMathOperator{\B}{B}

\begin{document}


\title{On the independence number of\\ $(3, 3)$-Ramsey graphs and the\\ Folkman number $F_e(3, 3; 4)$}

\author{
	Aleksandar Bikov\thanks{Corresponding author} \hfill Nedyalko Nenov\\
	\let\thefootnote\relax\footnote{Email addresses: \texttt{asbikov@fmi.uni-sofia.bg}, \texttt{nenov@fmi.uni-sofia.bg}}\\
	Faculty of Mathematics and Informatics\\
	Sofia University "St. Kliment Ohridski"\\
	5, James Bourchier Blvd.\\
	1164 Sofia, Bulgaria
}


\maketitle

\vspace{1em}

\begin{abstract}

The graph $G$ is called a $(3, 3)$-Ramsey graph if in every coloring of the edges of $G$ in two colors there is a monochromatic triangle. The minimum number of vertices of the $(3, 3)$-Ramsey graphs without 4-cliques is denoted by $F_e(3, 3; 4)$. The number $F_e(3, 3; 4)$ is referred to as the most wanted Folkman number. It is known that $20 \leq F_e(3, 3; 4) \leq 786$.

In this paper we prove that if $G$ is an $n$-vertex $(3, 3)$-Ramsey graph without 4-cliques, then $\alpha(G) \leq n - 16$, where $\alpha(G)$ denotes the independence number of $G$. Using the newly obtained bound on $\alpha(G)$ and complex computer calculations we obtain the new lower bound
$$F_e(3, 3; 4) \geq 21.$$
	
\end{abstract}

\vspace{2em}


\section{Introduction}

Only simple graphs are considered. Let $a_1, ..., a_s$ be positive integers. The symbol $G \arrowse (a_1, ..., a_s)$ ($G \arrowsv (a_1, ..., a_s)$) means that for every coloring of the edges (vertices) of the graph $G$ in $s$ colors there exist $i \in \{1, ..., s\}$ such that there is a monochromatic $a_i$-clique of color $i$. If $G \arrowse (3, 3)$ we say that $G$ is a $(3, 3)$-Ramsey graph.

Define:

$\mH_e(a_1, ..., a_s; q) = \{G : G \arrowse (a_1, ..., a_s) \mbox{ and } \omega(G) < q\}.$

$\mH_e(a_1, ..., a_s; q; n) = \{G : G \in \mH_e(a_1,  ..., a_s; q) \mbox{ and } |\V(G)| = n\}.$\\

The edge Folkman numbers $F_e(a_1, ..., a_s; q)$ are defined with the equality

$F_e(a_1, ..., a_s; q) = \min\{|\V(G)| : G \in \mH_e(a_1, ..., a_s; q)\},$

i.e. $F_e(a_1, ..., a_s; q)$ is the smallest positive integer $n$ for which $\mH_e(a_1, ..., a_s; q; n) \neq \emptyset$. These notations are first defined in \cite{LRU01}, where some important properties of the Folkman numbers are proved.\\

Folkman \cite{Fol70} proved in 1970 that

$\mH_e(a_1, a_2; q) \neq \emptyset \Leftrightarrow q \geq \max\{a_1, a_2\} + 1.$

Therefore, $F_e(3, 3; q)$ exists if and only if $q \geq 4$.

From $K_6 \arrowse (3, 3)$ it follows that $F_e(3, 3; q) = 6$ if $q \geq 7$. It is also known that

$F_e(3, 3; q) = \begin{cases}
		8, & \emph{if $q = 6$, \cite{Gra68}}\\
		15, & \emph{if $q = 5$, \cite{Nen81} and \cite{PRU99}}.
	\end{cases}$

The exact value of the number $F_e(3, 3; 4)$ is not yet computed. For now it is known that

$$20 \leq F_e(3, 3; 4) \leq 786, \cite{BN17} \cite{LRX14}.$$

In Table \ref{table: history of F_e(3, 3; 4)} are given the main stages in bounding $F_e(3, 3; 4)$

\begin{table}[h]
	\centering
	\resizebox{0.8\textwidth}{!}{
	\begin{tabular}{  c  c  l  }
		\hline
		\hline
		year	& lower/upper		& who/what													\\
				& bounds			& 															\\
		\hline
		1967	& any?				& Erd\H{o}s and Hajnal \cite{EH67}							\\
		1970	& exist				& Folkman \cite{Fol70}										\\
		1972	& \makebox[1em]{11} \textendash \makebox[5em]{} & Lin implicit in \cite{Lin72}, implied by $F_e(3, 3; 5) \geq 10$\\
		1975	& \makebox[1em]{} \textendash \makebox[5em]{ $10 \times 10^{10}?$}& Erd\H{o}s offers \$100 for proof \cite{Erd75}			\\
		1983	& \makebox[1em]{13} \textendash	\makebox[5em]{}			& implied by a result of Nenov \cite{Nen83}					\\
		1984	& \makebox[1em]{14} \textendash	\makebox[5em]{}			& implied by a result of Nenov \cite{Nen84}					\\
		1986	& \makebox[1em]{} \textendash \makebox[5em]{ $8 \times 10^{11}$} & Frankl and R\"odl \cite{FR86}							\\
		1988	& \makebox[1em]{} \textendash \makebox[5em]{ $3 \times 10^9$} & Spencer \cite{Spe88}										\\
		1999	& \makebox[1em]{16} \textendash\makebox[5em]{}				& Piwakowski, Radziszowski, and Urba\'{n}ski, implicit in \cite{PRU99}\\
		2007	& \makebox[1em]{19} \textendash	\makebox[5em]{}			& Radziszowski and Xu \cite{RX07}							\\
		2008	& \makebox[1em]{} \textendash \makebox[5em]{ 9697}			& Lu \cite{Lu08}											\\
		2008	& \makebox[1em]{} \textendash \makebox[5em]{ 941}				& Dudek and R\"odl \cite{DR08}								\\
		2012	& \makebox[1em]{} \textendash \makebox[5em]{ 100?}			& Graham offers \$100 for proof								\\
		2014	& \makebox[1em]{} \textendash \makebox[5em]{ 786}				& Lange, Radziszowski, and, Xu \cite{LRX14}					\\
		2017	& \makebox[1em]{20} \textendash	\makebox[5em]{}			& Bikov and Nenov \cite{BN17}								\\
		\hline
	\end{tabular}
	}
	\caption{History of the Folkman number $F_e(3, 3; 4)$ from \cite{KWR18}}
	\label{table: history of F_e(3, 3; 4)}
\end{table}

More information about the numbers $F_e(3, 3; q)$ can be found in \cite{Gra12}, \cite{KWR18}, \cite{LRX14} and \cite{Soi08}. As seen on Table \ref{table: history of F_e(3, 3; 4)}, the number $F_e(3, 3; 4)$ is very hard to bound and it is the most searched Folkman number. The reason for this is that we know very little about the graphs in $\mH_e(3, 3; 4)$.

In this work we give an upper bound on the independence number of the graphs in $\mH_e(3, 3; 4)$ by proving the following

\begin{theorem}
\label{theorem: alpha(G) leq n - 16}
Let $G \in \mH_e(3, 3; 4; n)$. Then
$$\alpha(G) \leq n - 16.$$
\end{theorem}

With the help of computer calculations and Theorem \ref{theorem: alpha(G) leq n - 16} we improve the main result $F_e(3, 3; 4) \geq 20$ from \cite{BN17} by proving:

\begin{theorem}
\label{theorem: F_e(3, 3; 4) geq 21}
$F_e(3, 3; 4) \geq 21$.
\end{theorem}

\section{Some necessary properties of the graphs in $\mH_e(3, 3; q)$}

Many useful properties of the graphs in $\mH_e(3, 3; q)$ follow from the fact that homomorphism of graphs preserves the Ramsey properties. In our situation, this means
\begin{proposition}
\label{proposition: G' arrowse (3, 3)}
Let $G \overset{\phi}{\rightarrow} G'$ be a homomorphism and $G \arrowse (3, 3)$. Then, $G' \arrowse (3, 3)$.
\end{proposition}
\begin{proof}
Suppose the opposite is true and consider a 2-coloring of $\E(G')$ without monochromatic triangles. Define a 2-coloring of $\E(G)$ in the following way: the edge $[u, v]$ is colored in the same color as the edge $[\phi(u), \phi(v)]$. Clearly, this coloring of $\E(G)$ does not contain monochromatic triangles.
\end{proof}
In the general case, it is true that $G \arrowse (a_1, ..., a_s) \Rightarrow G' \arrowse (a_1, ..., a_s)$, and $G \arrowsv (a_1, ..., a_s) \Rightarrow G' \arrowsv (a_1, ..., a_s)$, which is proved in the same way.

Now consider the canonical homomorphism $G \overset{\phi}{\rightarrow} K_{\chi(G)}$. If $G \arrowse (3, 3)$, then $K_{\chi(G)} \arrowse (3, 3)$, and therefore
\begin{equation}
\label{equation: chi(G) geq 6}
\min\{\chi(G): G \in \mH_e(3, 3; q)\} \geq 6, \cite{Lin72}.
\end{equation}
For $q \geq 5$, the inequality (\ref{equation: chi(G) geq 6}) is exact. It is not known whether this inequality is exact in the case $q = 4$. This is a special case of a result of Lin \cite{Lin72}. In the general case, $G \arrowse (a_1, ..., a_s) \Rightarrow \chi(G) \geq R(a_1, ..., a_s)$, which is proved in the same way.

With $K_p + G$ we denote the graph obtained by connecting with an edge every vertex of $K_p$ to every vertex of $G$.

Further, we will need the following

\begin{proposition}
\label{proposition: K_1 + H arrowse (3, 3)}
Let $G \arrowse (3, 3)$, $A$ be an independent set of vertices of $G$, and $H = G - A$. Then, $K_1 + H \arrowse (3, 3)$.
\end{proposition}
\begin{proof}
Consider the mapping $G \overset{\phi}{\rightarrow} K_1 + H$:

$\phi(v) = \begin{cases}
		\V(K_1), & \emph{if $v \in A$}\\
		v, & \emph{if $v \in \V(H)$}.
	\end{cases}$\\	
It is clear that $\phi$ is a homomorphism, and according to Proposition \ref{proposition: G' arrowse (3, 3)}, $K_1 + H \arrowse (3, 3)$.
\end{proof}
The usefulness of Proposition \ref{proposition: K_1 + H arrowse (3, 3)} lies in the fact that the graph $G$ can be obtained by adding independent vertices to the smaller graph $H$. In the general case it is true that if $G \arrowse (a_1, ..., a_s)$, then $K_1 + H \arrowse (a_1, ..., a_s)$.

\begin{remark}
Other proof of Proposition \ref{proposition: K_1 + H arrowse (3, 3)} is given in the proof of Theorem \ref{theorem: F_e(3, 3; 4) geq 21} from \cite{RX07}. However, the proposition is not explicitly formulated.
\end{remark}

A topic of significant interest are homomorphisms in Proposition \ref{proposition: G' arrowse (3, 3)} which do not increase the clique number. They could be used to obtain non trivial results. For example, in \cite{HN79} a 20-vertex graph in $\mH_e(3, 3; 5)$ is constructed. Using a homomorphism, in the same work a 16-vertex graph in $\mH_e(3, 3; 5)$ is obtained from this graph. Thus, in 1978 the bound $F_e(3, 3; 5) \leq 16$ is proved, improving the previous result $F_e(3, 3; 5) \leq 18$ from 1973 \cite{Irw73}.

The graph $G$ is vertex-critical (edge-critical) in $\mH_e(3, 3; 4)$ if $G \in \mH_e(3, 3; 4)$ and $G - v \not\in \mH_e(3, 3; 4), \forall v \in \V(G)$ ($G - e \not\in \mH_e(3, 3; 4), \forall e \in \E(G)$). It is true that
\begin{equation}
\label{equation: delta(G) geq 8}
\min\{\delta(G): G \mbox{ is a vertex-critical graph in } \mH_e(3, 3; 4)\} \geq 8, \cite{Bik18}\cite{Bik16},
\end{equation}
where $\delta(G)$ is the minimum degree of $G$.
\begin{remark}
\label{remark: minimum degree of vertex-critical graphs}
In \cite{Bik18} and \cite{Bik16} (\ref{equation: delta(G) geq 8}) is formulated for edge-critical graphs without isolated vertices. The proof is obviously also true for vertex critical graphs, and therefore further we shall use (\ref{equation: delta(G) geq 8}).
\end{remark}

It is not known if the inequality (\ref{equation: delta(G) geq 8}) is exact.\\

\section{Auxiliary notation and propositions}

Let $G \in \mH_e(3, 3; 4)$, $A$ be an independent set of vertices of $G$, and $H_1 = G - A$. By Proposition \ref{proposition: K_1 + H arrowse (3, 3)}, $K_1 + H_1 \arrowse (3, 3)$. If $A_1$ is an independent set in $H_1$ and $H_2 = H_1 - A_1$, then $K_2 + H_2 \arrowse (3, 3)$. If $A_2$ is an independent set in $H_2$ and $H_3 = H_2 - A_2$, then $K_3 + H_3 \arrowse (3, 3)$, etc. This way, we obtain a sequence $G \supseteq H_1 \supseteq H_2 \supseteq H_3 \supseteq ...$, in which $\omega(H_i) \leq 3$ and $K_i + H_i \arrowse (3, 3)$. Further, in the proof of Theorem \ref{theorem: F_e(3, 3; 4) geq 21}, we will use such a sequence of graphs. Because of this, the following notations are convenient:\\

$\mL(n; p) = \set{G : \abs{\V(G)} = n, \omega(G) < 4 \mbox{ and } K_p + G \arrowse (3, 3)}$

$\mL(n; p; s) = \set{G \in \mL(n; p) : \alpha(G) = s}$\\

Obviously, $\mL(n; 0) = \mH_e(3, 3; 4; n)$. Let us note that

$\mL(n; 1) = \emptyset$, if $n \leq 13$, \cite{PRU99}.

$|\mL(14; 1)| = 153$, \cite{PRU99}.

$|\mL(15; 1)| = 2081234$, \cite{BN17}.

The graphs in $\mL(16; 1)$ are not known. In the proof of Theorem \ref{theorem: F_e(3, 3; 4) geq 21} we obtain some of the graphs in $\mL(16; 1)$. The graphs in $\mL(15; 1)$ will be used in the proofs of Theorem \ref{theorem: alpha(G) leq n - 16} and Theorem \ref{theorem: F_e(3, 3; 4) geq 21}. Some properties of the graphs in $\mL(14; 1)$, $\mL(15; 1)$, and $\mL(16; 1)$ are given in Table \ref{table: mL(14; 1) properties}, Table \ref{table: mL(15; 1) properties}, and Table \ref{table: mL_(+K_3)(16; 1; 4) properties}.

It is true that

$G \arrowsv (3, 3) \Rightarrow K_1 + G \arrowse (3, 3)$, (\cite{Irw73}, see acknowledgments).

This fact is known as Posa's implication. Posa uses this implication to prove that $\mH_e(3, 3; 5) \neq \emptyset$ (unpublished). Irwing \cite{Irw73} uses the implication to prove the bound $\mH_e(3, 3; 5) \leq 18$. If additionally $\omega(G) = 3$, then $G \in \mL(n, 1)$. In \cite{PRU99} it is proved that if $G \in \mL(14, 1)$, then $G \arrowsv (3, 3)$. This result was used in \cite{RX07} to obtain the bound $F_e(3, 3; 4) \geq 19$. There exist, however, graphs $G$ in $\mL(15, 1)$ which do not have the property $G \arrowsv (3, 3)$. There are 20 such graphs and they are obtained in \cite{BN17}. Furthermore, these graphs do not have the property $G \arrowsv (2, 2, 3)$. This is one of the reasons why the method in the proof of $F_e(3, 3; 4) \geq 19$ in \cite{RX07} is inapplicable for proving $F_e(3, 3; 4) \geq n$, $n \geq 20$.

By Proposition \ref{proposition: K_1 + H arrowse (3, 3)}, if $G \in \mL(n, 0)$ and $A$ is an independent set of vertices of $G$, then $G - A \in \mL(n - |A|, 1)$. In \cite{BN17} we formulate without proof the following generalization of this fact:
\begin{proposition}
	\label{proposition: H in mL(n - abs(A), p + 1)}
	\cite{BN17}
	Let $G \in \mL(n; p)$, $A \subseteq \V(G)$ be an independent set of vertices of $G$ and $H = G - A$. Then $H \in \mL(n - \abs{A}; p + 1)$.
\end{proposition}

\begin{proof}
Since, $G \in \mL(n; p)$, $K_p + G \arrowse (3, 3)$. According to Proposition \ref{proposition: K_1 + H arrowse (3, 3)}, $K_1 + ((K_p + G) - A) \arrowse (3, 3)$. Since $(K_p + G) - A = K_p + (G - A) = K_p + H$ and $K_1 + (K_p + H) = K_{p + 1} + H$, we obtain $K_{p + 1} + H \arrowse (3, 3)$. Thus, $H \in \mL(n - \abs{A}; p + 1)$.
\end{proof}

We denote by $\mL_{max}(n; p; s)$ the set of all maximal $K_4$-free graphs in $\mL(n; p; s)$, i.e. the graphs $G \in \mL(n; p; s)$ for which $\omega(G + e) = 4$ for every $e \in \E(\overline{G})$.  

The graph $G$ is called a $(+K_3)$-graph if $G + e$ contains a new $3$-clique for every $e \in \E(\overline{G})$. Clearly, $G$ is a $(+K_3)$-graph if and only if $N(u) \cap N(v) = \emptyset$ for every pair of non-adjacent vertices $u$ and $v$ of $G$, i.e. either $G$ is a complete graph or the diameter of $G$ is equal to 2. The set of all $(+K_3)$-graphs in $\mL(n; p; s)$ is denoted by $\mL_{+K_3}(n; p; s)$. Obviously, $\mL_{max}(n; p; s) \subseteq \mL_{+K_3}(n; p; s)$.\\

\noindent For convenience, we will also use the following notations:

$\mL_{max}(n; p; \leq s) = \bigcup_{s' \leq s}\mL_{max}(n; p; s')$

$\mL_{+K_3}(n; p; \leq s) = \bigcup_{s' \leq s}\mL_{+K_3}(n; p; s')$\\

It is easy to see that if $G$ is a maximal $K_4$-free graph and $A$ is an independent set of vertices in $G$, then $G - A$ is a $(+K_3)$-graph. Because of this, regarding the graphs in $\mL_{max}(n; p; s)$, from Proposition \ref{proposition: H in mL(n - abs(A), p + 1)} it follows easily that

\begin{proposition}
	\label{proposition: H in mL_(+K_3)(n - s; p + 1; leq s)}
	\cite{BN17}
	Let $G \in \mL_{max}(n; p; s)$. Let $A \subseteq \V(G)$ be an independent set of vertices of $G$, $\abs{A} = s$ and $H = G - A$. Then,
	
	$H \in \mL_{+K_3}(n - s; p + 1; \leq s)$.
\end{proposition}

Further, the bound $F_e(3, 3; 4) \geq 21$ will be proved with the help of Algorithm \ref{algorithm: finding mL_(max)(n; p; s)} and Algorithm \ref{algorithm: finding mL_(max)(n; 0; s)}, which are based on Proposition \ref{proposition: H in mL_(+K_3)(n - s; p + 1; leq s)}.

\begin{definition}
\label{definition: Sperner graph}
The graph $G$ is called a Sperner graph if $N_G(u) \subseteq N_G(v)$ for some pair of vertices $u, v \in \V(G)$.
\end{definition}

Let $G \in \mL(n; p; s)$ and $N_G(u) \subseteq N_G(v)$. Then, $K_p + (G - u)$ is a homomorphic image of $K_p + G$ and by Proposition \ref{proposition: G' arrowse (3, 3)}, $K_p + (G - u) \arrowse (3, 3)$, i.e. $G - u \in \mL(n - 1; p; s')$, $s - 1 \leq s' \leq s$. Therefore, every Sperner graph $G \in \mL(n; p; s)$ is obtained by adding one new vertex to some graph $H \in \mL(n - 1; p; s')$, $s - 1 \leq s' \leq s$. In the special case, when $G$ is a Sperner graph and $G \in \mL_{max}(n; p; s)$, from $N_G(u) \subseteq N_G(v)$ it follows that $N_G(u) = N_G(v)$ and $G - u \in \mL_{max}(n - 1; p; s')$, $s - 1 \leq s' \leq s$. Hence, it is true

\begin{proposition}
\label{proposition: if G is a Sperner graph, then G is obtained by duplicating a vertex in H}
If $G \in \mL_{max}(n; p; s)$ is a Sperner graph, then $G$ is obtained by duplicating a vertex in some graph $H \in \mL_{max}(n - 1; p; s')$, $s - 1 \leq s' \leq s$.
\end{proposition}

From (\ref{equation: chi(G) geq 6}) and $K_p + G \arrowse (3, 3)$ it follows that
\begin{equation}
\label{equation: chi(G) geq 6 - p}
G \in \mL(n; p) \Rightarrow \chi(G) \geq 6 - p.
\end{equation}

We will use this fact in Algorithm \ref{algorithm: finding mL_(max)(n; p; s)}.

\section{Proof of Theorem \ref{theorem: alpha(G) leq n - 16}}

\begin{definition}
\label{definition: B(H)}
For every graph $H$ denote by $\mM(H)$ the set of all maximal $K_3$-free subsets of $\V(H)$. Let
$$\mM(H) = \{M_1, ..., M_k\}.$$
We denote by $\B(H)$ the graph which is obtained by adding to $H$ new independent vertices $u_1, ..., u_k$ and new edges incident to $u_1, ..., u_k$ such that
$$N_{\B(H)}(u_i) = M_i, \ i = 1, ..., k.$$
\end{definition}

\begin{lemma}
\label{lemma: If G arrowse (3, 3), then B(H) arrowse (3, 3)}
Let $G$ be a graph, $\omega(G) = 3$, $A$ be an independent set of vertices of $G$, and $H = G - A$. If $G \arrowse (3, 3)$, then $\B(H) \arrowse (3, 3)$. 
\end{lemma}

\begin{proof}
Let $\mM(H) = \{M_1, ..., M_k\}$ be the same as in Definition \ref{definition: B(H)} and $A = \{v_1, ..., v_s\}$. Let $v_i \in A$. Then, $N_G(v_i) \subseteq M_j$ for some $j \in \{1, ..., k\}$. Let $j_i$ be the smallest index for which $N_G(v_i) \subseteq M_{j_i}$. We define a supergraph $\widetilde{G}$ of $G$ in the following way: for each $v_i \in A$ we add to $\E(G)$ the new edges $[v_i, u], u \in M_{j_i} \setminus N_G(v_i)$. Clearly, $\V(\widetilde{G}) = \V(G)$, $A$ is an independent set of vertices of $\widetilde{G}$, $\widetilde{G} - A = H$ and
$$N_{\widetilde{G}}(v_i) \in \mM(H), i = 1, ..., s.$$
Since $\widetilde{G}$ is a subgraph of $G$, it follows
\begin{equation}
\label{equation: widetilde(G) arrowse (3, 3)}
\widetilde{G} \arrowse (3, 3).
\end{equation}

If $\{N_{\widetilde{G}}(v_1), ..., N_{\widetilde{G}}(v_s)\}$ is a subset of $\mM(H)$, then $\widetilde{G}$ is a subgraph of $\B(H)$ and from (\ref{equation: widetilde(G) arrowse (3, 3)}) it follows that $\B(H) \arrowse (3, 3)$.

Let $\{N_{\widetilde{G}}(v_1), ..., N_{\widetilde{G}}(v_s)\}$ be a multiset and $N_{\widetilde{G}}(v_i) = N_{\widetilde{G}}(v_j), \ i < j$. Let $\widetilde{G'}$ = $\widetilde{G} - v_j$. Obviously, $\widetilde{G'}$ is a homomorphic image of $G$. Therefore, from Proposition \ref{proposition: G' arrowse (3, 3)} and (\ref{equation: widetilde(G) arrowse (3, 3)}) it follows that $\widetilde{G'} \arrowse (3, 3)$.

If in $\{N_{\widetilde{G'}}(v_i) | i = 1, ... s, i \neq j\}$ there is also a duplication, then in the same way we remove from $\widetilde{G'}$ one of the duplicating vertices among $v_1, ..., v_{i - 1}, v_{i + 1}, ..., v_s$ and we obtain a graph $\widetilde{G''}$ such that $\widetilde{G''} \arrowse (3, 3)$.

In the end, a graph $\widetilde{\widetilde{G}}$ is reached such that $H = \widetilde{\widetilde{G}} - A'$, where $A' \subseteq A$, $\{N_{\widetilde{\widetilde{G}}}(v) | v \in A'\}$ is a subset of $\mM(H)$, and $\widetilde{\widetilde{G}} \arrowse (3, 3)$. Since $\widetilde{\widetilde{G}}$ is a subgraph of $\B(H)$, it follows that $\B(H) \arrowse (3, 3)$.
\end{proof}

\textit{Proof of }\textbf{Theorem \ref{theorem: alpha(G) leq n - 16}.} Suppose the opposite is true, i.e. $\alpha(G) \geq n - 15$. Let $A = \{v_1, ..., v_{n - 15}\}$ be an independent set of vertices of $G$, and $H = G - A$. Then, from Lemma \ref{lemma: If G arrowse (3, 3), then B(H) arrowse (3, 3)} it follows $\B(H) \arrowse (3, 3)$.

According to Proposition \ref{proposition: H in mL(n - abs(A), p + 1)} ($p = 0$), $H \in \mL(15; 1)$. All 2081234 graphs in $\mL(15; 1)$ were obtained in (\cite{BN17}, Remark 4.4). With a computer we check that $\B(H) \not\arrowse (3, 3)$ for all $H \in \mL(15, 1)$. This contradiction proves the theorem.
\qed\\

\begin{corollary}
\label{corollary: F_e(3, 3; 4) geq 20}
\cite{BN17}
$F_e(3, 3; 4) \geq 20$.
\end{corollary}
\begin{proof}
Suppose that $G$ is a 19-vertex $(3, 3)$-Ramsey graph and $\omega(G) = 3$. From Theorem \ref{theorem: alpha(G) leq n - 16} it follows that $\alpha(G) \leq 3$. This contradicts with the equality $R(4, 4) = 18$.
\end{proof}

The graphs $\B(H), H \in \mL(15; 1)$, in the proof of Theorem \ref{theorem: alpha(G) leq n - 16} have between 50 and 210 vertices. We used the {\em zchaff} SAT solver \cite{zchaff} to prove that these graphs are not $(3, 3)$-Ramsey graphs. The problem of determining if a graph $G$ is a $(3, 3)$-Ramsey graph can be transformed into a problem of satisfiability of a boolean formula $\phi_G$ in conjunctive normal form with $|\E(G)|$ variables. Let $e_1, ..., e_{|\E(G)|}$ be the edges of $G$ and $x_1, ..., x_{|\E(G)|}$ be the corresponding boolean variables in $\phi_G$. For every triangle in $G$ formed by the edges $e_i$$e_j$$e_k$, we add two clauses to $\phi_G$
$$(x_i \lor x_j \lor x_k) \land (\overline{x_i} \lor \overline{x_j} \lor \overline{x_k}).$$
It is easy to see that $G \arrowse (3, 3)$ if and only if $\phi_G$ is satisfiable.

Even though the  graphs $\B(H), H \in \mL(15; 1)$, have up to 210 vertices, SAT solvers are able to solve the resulting boolean formulas in a short amount of time. There exist smaller graphs $G$ for which it is difficult to determine if $G \arrowse (3, 3)$. For example, Exoo conjectured that the 127-vertex graph $G_{127}$, used by Hill and Irwing \cite{HI82} to prove the bound $R(4, 4, 4) \geq 128$, has the property $G_{127} \arrowse (3, 3)$. This conjecture was studied in \cite{RX07} and \cite{KWR18}. It is still unknown whether $G_{127} \arrowse (3, 3)$.

\section{Proof of Theorem \ref{theorem: F_e(3, 3; 4) geq 21}}

According to Proposition \ref{proposition: if G is a Sperner graph, then G is obtained by duplicating a vertex in H}, all Sperner graphs in $\mL_{max}(n; p; s)$ can be obtained easily by duplicating a vertex in graphs in $\mL_{max}(n - 1; p; s'), s -1 \leq s' \leq s$. By Proposition \ref{proposition: H in mL_(+K_3)(n - s; p + 1; leq s)}, the non-Sperner graphs in $\mL_{max}(n; p; s)$ are obtained by adding $s$ independent vertices to some graphs in $\mL_{+K_3}(n - s; p + 1; \leq s)$. This is realized with the help of the following algorithm:

\begin{algorithm}
	\label{algorithm: finding mL_(max)(n; p; s)}
	\cite{BN17}
	Finding all non-Sperner graphs in $\mL_{max}(n; p; s)$ for fixed $n$, $p$ and $s$.
	
	\emph{1.} The input of the algorithm is the set $\mathcal{A} = \mL_{+K_3}(n - s; p + 1; \leq s)$. The output will be the set $\mathcal{B}$ of all non-Sperner graphs in $\mL_{max}(n; p; s)$. Initially, set $\mathcal{B} = \emptyset$.
	
	\emph{2.} For each graph $H \in \mathcal{A}$:
	
	\emph{2.1.} Find the family $\mathcal{M}(H) = \set{M_1, ..., M_t}$ of all maximal $K_3$-free subsets of $\V(H)$.
	
	\emph{2.2.} Find all $s$-element subsets $N = \set{M_{i_1}, M_{i_2}, ..., M_{i_s}}$ of $\mathcal{M}(H)$, which fulfill the conditions:
	
	(a) $M_{i_j} \neq N_H(v)$ for every $v \in \V(H)$ and for every $M_{i_j} \in N$.
	
	(b) $K_2 \subseteq M_{i_j} \cap M_{i_k}$ for every $M_{i_j}, M_{i_k} \in N$.
	
	(c) $\alpha(H - \bigcup_{M_{i_j} \in N'} M_{i_j}) \leq s - \abs{N'}$ for every $N' \subseteq N$.
	
	\emph{2.3.} For each of the found in 2.2 $s$-element subsets $N = \set{M_{i_1}, M_{i_2}, ..., M_{i_s}}$ of $\mathcal{M}(H)$ construct the graph $G = G(N)$ by adding new independent vertices $v_1, v_2, ..., v_s$ to $\V(H)$ such that $N_G(v_j) = M_{i_j}, j = 1, ..., s$. If $G$ is not a Sperner graph and $\omega(G + e) = 4, \forall e \in \E(\overline{G})$, then add $G$ to $\mathcal{B}$.
	
	\emph{3.} Remove the isomorph copies of graphs from $\mathcal{B}$.
	
	\emph{4.} Remove from $\mathcal{B}$ all graphs with chromatic number less than $6 - p$.
	
	\emph{5.} Remove from $\mathcal{B}$ all graphs $G$ for which $K_p + G \not\arrowse (3, 3)$.
\end{algorithm}

\begin{theorem}
	\label{theorem: finding mL_(max)(n; p; s)}
	\cite{BN17}
	After the execution of Algorithm \ref{algorithm: finding mL_(max)(n; p; s)}, the obtained set $\mathcal{B}$ coincides with the set of all non-Sperner graphs in $\mL_{max}(n; p; s)$.
\end{theorem}

The correctness of Algorithm \ref{algorithm: finding mL_(max)(n; p; s)} is guarantied by the proof of Theorem \ref{theorem: finding mL_(max)(n; p; s)} in \cite{BN17}. Here we will only note some details. The condition (a) has to be satisfied since $G = G(N)$ is not a Sperner graph. The condition (b) follows from the maximality of $G = G(N)$. Even if both conditions (a) and (b) are satisfied, additional checks in step 2.3 are still needed to make sure that only maximal non-Sperner graphs are added to $\mB$. From condition (c) it follows that only graphs with independence number $s$ are added to $\mB$. If $N' = \emptyset$, then (c) clearly holds, since $\alpha(H) \leq s$. If $|N'| = 1$, then for each added vertex $v_j$ it is guarantied that $v_j$ does not form an independent set with $s$ vertices of $H$. If $|N'| = 2$, then for every two added vertices $v_j, v_k$ it is guarantied that $v_j$ and $v_k$ do not form an independent set with $(s - 1)$ vertices of $H$, etc. The graphs in $\mB$ must satisfy the condition in step 4 according to (\ref{equation: chi(G) geq 6 - p}).\\

According to Corollary \ref{corollary: F_e(3, 3; 4) geq 20}, the graphs in $\mL(20; 0)$ are vertex critical, and by (\ref{equation: delta(G) geq 8}), these graphs have minimal degree greater than or equal to 8. Because of this, we modify Algorithm \ref{algorithm: finding mL_(max)(n; p; s)} in the case $p = 0$ in such a way that only graphs $G$ with $\delta(G) \geq 8$ are added to $\mB$:

\begin{algorithm}
	\label{algorithm: finding mL_(max)(n; 0; s)}
	\cite{BN17}
	Optimization of Algorithm \ref{algorithm: finding mL_(max)(n; p; s)} for finding all non-Sperner graphs $G \in \mL_{max}(n; 0; s)$ with $\delta(G) \geq 8$.

1. In step 1 we remove from the set $\mathcal{A}$ the graphs with minimum degree less than 8 - s.

2. In step 2.2 we add the following conditions for the subset $N$:

(d) $\abs{M_{i_j}} \geq 8$ for every $M_{i_j} \in N$.

(e) If $N' \subseteq N$, then $d_H(v) \geq 8 - s + \abs{N'}$ for every $v \not\in \bigcup_{M_{i_j} \in N'} M_{i_j}$.

\end{algorithm}

\vspace{2em}

\textit{Proof of }\textbf{Theorem \ref{theorem: F_e(3, 3; 4) geq 21}.} Suppose the opposite is true and let $G$ be a $20$-vertex maximal $(3, 3)$-Ramsey graph with $\omega(G) = 3$. From Theorem \ref{theorem: alpha(G) leq n - 16} it follows that $\alpha(G) \leq 4$. Now, from $R(4, 4) = 18$ it follows that $\alpha(G) = 4$. Therefore, it is enough to prove that $\mL_{max}(20; 0; 4) = \emptyset$. First, we will successively obtain all graphs in the sets $\mL_{+K_3}(8; 3; \leq 4)$, $\mL_{+K_3}(12; 2; \leq 4)$, and $\mL_{+K_3}(16; 1; \leq 4)$.\\

\textit{Obtaining all graphs in $\mL_{+K_3}(8; 3; \leq 4)$:}

We use the {\em geng} tool included in the {\em nauty} programs \cite{MP13} to generate All non-isomorphic graphs of order 8 are generated using the {\em geng} tool included in the {\em nauty} programs \cite{MP13}. Among them we find all 1178 graphs in $\mL_{+K_3}(8; 3; \leq 4)$ (see Table \ref{table: mL_(+K_3)(8; 3; leq 4) properties}).\\

\textit{Obtaining all graphs in $\mL_{+K_3}(12; 2; \leq 4)$:}

From $R(3, 4) = 9$ it follows that $\mL(12; 2; \leq 2) = \emptyset$. All 1449166 12-vertex graphs $G$ with $\omega(G) < 4$ and $\alpha(G) < 4$ are known and available on \cite{McK_r}. Among them there are 321 graphs in $\mL_{max}(12; 2; 3)$. We use {\em geng} to generate all non-isomorphic graphs of order 11. Among them we find all 372 graphs in $\mL_{max}(11; 2; \leq 4)$.  According to Proposition \ref{proposition: if G is a Sperner graph, then G is obtained by duplicating a vertex in H}, all Sperner graphs in $\mL_{max}(12; 2; 4)$ are obtained by duplicating a vertex in some of the graphs in $\mL_{max}(11; 2; \leq 4)$. This way, we find all 1341 Sperner graphs in $\mL_{max}(12; 2; 4)$. We execute Algorithm \ref{algorithm: finding mL_(max)(n; p; s)} ($n = 12, p = 2, s = 4$) with input the set $\mA = \mL_{+K_3}(8; 3; \leq 4)$ to output all 815 non-Sperner graphs in $\mL_{max}(12; 2; 4)$. Thus, $|\mL_{max}(12; 2; \leq 4)| = 2477$. By removing edges from the graphs in $\mL_{max}(12; 2; \leq 4)$ we find all 539410034 graphs in $\mL_{+K_3}(12; 2; \leq 4)$ (see Table \ref{table: mL_(+K_3)(12; 2; leq 4) properties}).\\

\textit{Obtaining all graphs in $\mL_{+K_3}(16; 1; \leq 4)$:}

From $R(3, 4) = 9$ it follows that $\mL(16; 1; \leq 2) = \emptyset$. There are only two 16-vertex graphs $G$ such that $\omega(G) < 4$ and $\alpha(G) < 4$ \cite{McK_r}. We check with a computer that none of them belongs to $\mL(16; 1)$, and therefore $\mL(16; 1; 3) = \emptyset$. Thus, $\mL(16; 1; \leq 4) = \mL(16; 1; 4)$. All 5772 graphs in $\mL_{max}(15; 1; \leq 4)$ were obtained in part 1 of the proof of the Main Theorem in \cite{BN17}. According to Proposition \ref{proposition: if G is a Sperner graph, then G is obtained by duplicating a vertex in H}, all Sperner graphs in $\mL_{max}(16; 1; 4)$ are obtained by duplicating a vertex in some of the graphs in $\mL_{max}(15; 1; \leq 4)$. This way, we find all 21749 Sperner graphs in $\mL_{max}(16; 1; 4)$. We execute Algorithm \ref{algorithm: finding mL_(max)(n; p; s)} ($n = 16, p = 1, s = 4$) with input the set $\mA = \mL_{+K_3}(12; 2; \leq 4)$ to output all 1676267 non-Sperner graphs in $\mL_{max}(16; 1; 4)$. Thus, $|\mL_{max}(16; 1; \leq 4)| = 1698016$. By removing edges from the graphs in $\mL_{max}(16; 1; \leq 4)$ we find all 3892126874 graphs in $\mL_{+K_3}(16; 1; \leq 4)$ (see Table \ref{table: mL_(+K_3)(16; 1; 4) properties}).\\

\textit{Proving that $\mL_{max}(20; 0; 4) = \emptyset$:}

We execute Algorithm \ref{algorithm: finding mL_(max)(n; 0; s)} ($n = 20, p = 0, s = 4$) with input the set $\mA = \mL_{+K_3}(16; 1; \leq 4)$. After the completion of step 4, 19803568 graphs remain in the set $\mB$ (see Table \ref{table: 20-vertex 6-chromatic properties}). None of these graphs satisfies the condition in step 5, hence $\mB = \emptyset$. We obtained that there are no non-Sperner graphs in $\mL_{max}(20; 0; 4)$ with minimum degree greater than or equal to 8. According to Corollary \ref{corollary: F_e(3, 3; 4) geq 20}, all graphs in $\mL_{max}(20; 0; 4)$ must be vertex critical. Therefore, there are no Sperner graphs in $\mL_{max}(20; 0; 4)$, and by (\ref{equation: delta(G) geq 8}), no graphs with minimum degree less than 8. We proved that $\mL_{max}(20; 0; 4) = \emptyset$, which finishes the proof.	
\qed\\

Some properties of the graphs in $\mL_{+K_3}(8; 3; \leq 4)$, $\mL_{+K_3}(12; 2; \leq 4)$, and $\mL_{+K_3}(16; 1; 4)$ are given in Table \ref{table: mL_(+K_3)(8; 3; leq 4) properties}, Table \ref{table: mL_(+K_3)(12; 2; leq 4) properties}, and Table \ref{table: mL_(+K_3)(16; 1; 4) properties}. Properties of the 20-vertex graphs obtained after the completion of step 4 of Algorithm \ref{algorithm: finding mL_(max)(n; 0; s)} $(n = 20, p = 0, s = 4)$ are given in Table \ref{table: 20-vertex 6-chromatic properties}.\\

All computations were performed on a personal computer. The most time consuming part of the proof was obtaining all graphs in $\mL_{+K_3}(16; 1; 4)$ by removing edges from the graphs in $\mL_{max}(16; 1; 4)$, which took about 4 months. After that, executing Algorithm \ref{algorithm: finding mL_(max)(n; 0; s)} $(n = 20, p = 0, s = 4)$ with input the graphs in $\mL_{+K_3}(16; 1; 4)$ was done in under 2 months.

In order to check the correctness of our computer programs implementing Algorithm \ref{algorithm: finding mL_(max)(n; p; s)}, we reproduced the 153 graphs in $\mL(14; 1)$, which were first obtained in \cite{PRU99} in a different way. Among the graphs in $\mL(14; 1)$ there are 8 maximal graphs, all of which have independence number 4, i.e. $|\mL_{max}(14; 1)| = 8$ and $\mL_{max}(14; 1) = \mL_{max}(14; 1; 4)$. Using {\em nauty} we obtained all 547524 graphs in $\mL_{+K_3}(10; 2; \leq 4)$. By executing Algorithm \ref{algorithm: finding mL_(max)(n; p; s)} ($n = 14, p = 1, s = 4$) with input $\mA = \mL_{+K_3}(10; 2; \leq 4)$ we found all 8 graphs in $\mL_{max}(14; 1; 4)$. By removing edges from the graphs in $\mL_{max}(14; 1)$ we obtained the 153 graphs in $\mL(14; 1)$.

\section{Concluding remarks}

In this section we consider the possibilities for improving the inequality $F_e(3, 3; 4) \geq 21$. We suppose the following conjecture is true:
\begin{conjecture}
	\label{conjecture: alpha(G) geq 5}
	$\min\{\alpha(G) : G \in \mH_e(3, 3; 4)\} \geq 5$.
\end{conjecture}
If $G \in \mH_e(3, 3; 4; n)$, $n \geq 25$, according to the equality $R(4, 5) = 25$ we have $\alpha(G) \geq 5$. All 24-vertex graphs with independence number 4 and clique number 3 are obtained in \cite{AM18}. With the help of a computer we check that none of these graphs belongs to $\mH_e(3, 3; 4)$. This way, we proved that if $G \in \mH_e(3, 3; 4; n)$, $n \geq 24$, then $\alpha(G) \geq 5$. To prove the conjecture it remains to consider the cases $n = 21$, $22$, and $23$.

By similar reasoning as in the proof of Theorem \ref{theorem: alpha(G) leq n - 16}, but with more calculations, it could be proved that
$$\alpha(G) \leq n - 17.$$
From this inequality and Conjecture \ref{conjecture: alpha(G) geq 5} it follows that $F_e(3, 3; 4) \geq 22$.


\vspace{8em}
{ ACKNOWLEDGMENTS}\\

We would like to thank Professor Stanis{\l}aw Radziszowski for his useful remarks and suggestions which led to the improvement of this paper.

The first author was supported by the National program "Young scientists and Postdoctoral candidates" of the Ministry of Education and Science.

\clearpage

\clearpage

\begin{table}
	\centering
	\resizebox{0.7\textwidth}{!}
	{
		\begin{tabular}{ | l r | l r | l r | l r | l r | l r | }
			\hline
			\multicolumn{2}{|c|}{\hbox to 2.5cm{$|\E(G)|$ \hfill $\#$}}&\multicolumn{2}{|c|}{\hbox to 2.5cm{$\delta(G)$ \hfill $\#$}}&\multicolumn{2}{|c|}{\hbox to 2.5cm{$\Delta(G)$ \hfill $\#$}}& \multicolumn{2}{|c|}{\hbox to 2.5cm{$\alpha(G)$ \hfill $\#$}}\\
			\hline
			42			&  1		& 4			& 91		& 7			& 3			& 4			& 111		\\
			43			&  2		& 5			& 58		& 8			& 90		& 5			& 39		\\
			44			&  7		& 6			& 4			& 10		& 60		& 6			& 2			\\
			45			&  20		& 			&			& 			& 			& 7			& 1			\\
			46			&  37		& 			& 			& 			& 			& 			& 			\\
			47			&  45		& 			& 			& 			& 			& 			& 			\\
			48			&  28		& 			& 			& 			& 			& 			& 			\\
			49			&  11		& 			& 			& 			& 			& 			& 			\\
			50			&  2		& 			& 			& 			& 			& 			& 			\\
			\hline
		\end{tabular}
	}
	\caption{Some properties of the graphs in $\mL(14; 1)$ obtained in \cite{PRU99}}
	\label{table: mL(14; 1) properties}
\vspace{4em}
	\centering
	\resizebox{0.7\textwidth}{!}
	{
		\begin{tabular}{ | l r | l r | l r | l r | l r | l r | }
			\hline
			\multicolumn{2}{|c|}{\hbox to 2.5cm{$|\E(G)|$ \hfill $\#$}}&\multicolumn{2}{|c|}{\hbox to 2.5cm{$\delta(G)$ \hfill $\#$}}&\multicolumn{2}{|c|}{\hbox to 2.5cm{$\Delta(G)$ \hfill $\#$}}& \multicolumn{2}{|c|}{\hbox to 2.5cm{$\alpha(G)$ \hfill $\#$}}\\
			\hline
			42			&  1		& 0			& 153		& 7			& 65		& 3			& 5			\\
			43			&  4		& 1			& 1 629		& 8			& 675 118	& 4			& 1 300 452	\\
			44			&  44		& 2			& 10 039	& 9			& 1 159 910	& 5			& 747 383	\\
			45			&  334		& 3			& 34 921	& 10		& 165 612	& 6			& 32 618	\\
			46			&  2 109	& 4			& 649 579	& 11		& 80 529	& 7			& 766		\\
			47			&  9 863	& 5			& 1 038 937	& 			& 			& 8			& 10		\\
			48			&  35 812	& 6			& 339 395	& 			& 			& 			& 			\\
			49			&  101 468	& 7			& 6 581		& 			& 			& 			& 			\\
			50			&  223 881	& 			& 			& 			& 			& 			& 			\\
			51			&  378 614	& 			& 			& 			& 			& 			& 			\\
			52			&  478 582	& 			& 			& 			& 			& 			& 			\\
			53			&  436 693	& 			& 			& 			& 			& 			& 			\\
			54			&  273 824	& 			& 			& 			& 			& 			& 			\\
			55			&  110 592	& 			& 			& 			& 			& 			& 			\\
			56			&  26 099	& 			& 			& 			& 			& 			& 			\\
			57			&  3 150	& 			& 			& 			& 			& 			& 			\\
			58			&  160		& 			& 			& 			& 			& 			& 			\\
			59			&  4		& 			& 			& 			& 			& 			& 			\\
			\hline
		\end{tabular}
	}
	\caption{Some properties of the graphs in $\mL(15; 1)$ obtained in \cite{BN17}}
	\label{table: mL(15; 1) properties}
\end{table}

\clearpage

\begin{table}
	\centering
	\resizebox{0.7\textwidth}{!}
	{
		\begin{tabular}{ | l r | l r | l r | l r | l r | l r | }
			\hline
			\multicolumn{2}{|c|}{\hbox to 2.5cm{$|\E(G)|$ \hfill $\#$}}&\multicolumn{2}{|c|}{\hbox to 2.5cm{$\delta(G)$ \hfill $\#$}}&\multicolumn{2}{|c|}{\hbox to 2.5cm{$\Delta(G)$ \hfill $\#$}}& \multicolumn{2}{|c|}{\hbox to 2.5cm{$\alpha(G)$ \hfill $\#$}}\\
			\hline
			10			&  1		& 1			& 15		& 3			& 2			& 2			& 3			\\
			11			&  3		& 2			& 552		& 4			& 108		& 3			& 705		\\
			12			&  28		& 3			& 560		& 5			& 610		& 4			& 470		\\
			13			&  114		& 4			& 49		& 6			& 387		& 			& 			\\
			14			&  258		& 5			& 2			& 7			& 71		& 			& 			\\
			15			&  328		& 			& 			& 			& 			& 			& 			\\
			16			&  253		& 			& 			& 			& 			& 			& 			\\
			17			&  127		& 			& 			& 			& 			& 			& 			\\
			18			&  47		& 			& 			& 			& 			& 			& 			\\
			19			&  14		& 			& 			& 			& 			& 			& 			\\
			20			&  4		& 			& 			& 			& 			& 			& 			\\
			21			&  1		& 			& 			& 			& 			& 			& 			\\
			\hline
		\end{tabular}
	}
	\caption{Some properties of the graphs in $\mL(8; 3; \leq 4)$}
	\label{table: mL_(+K_3)(8; 3; leq 4) properties}
	\vspace{4em}
	\centering
	\resizebox{0.7\textwidth}{!}
	{
		\begin{tabular}{ | l r | l r | l r | l r | l r | l r | }
			\hline
			\multicolumn{2}{|c|}{\hbox to 2.5cm{$|\E(G)|$ \hfill $\#$}}&\multicolumn{2}{|c|}{\hbox to 2.5cm{$\delta(G)$ \hfill $\#$}}&\multicolumn{2}{|c|}{\hbox to 2.5cm{$\Delta(G)$ \hfill $\#$}}& \multicolumn{2}{|c|}{\hbox to 2.5cm{$\alpha(G)$ \hfill $\#$}}\\
			\hline
			23			&  5		& 2			& 3271422	& 5			& 449820	& 3			& 1217871	\\
			24			&  231		& 3			& 200573349	& 6			& 90348516	& 4			& 538192163	\\
			25			&  10970	& 4			& 317244496	& 7			& 326214208	& 			& 			\\
			26			&  254789	& 5			& 18296860	& 8			& 113842493	& 			& 			\\
			27			&  2675686	& 6			& 23902		& 9			& 8451810	& 			&			\\
			28			&  14355266	& 7			& 5			& 10		& 103082	& 			& 			\\
			29			&  44690777	& 			& 			& 11		& 105		& 			& 			\\
			30			&  88716906	& 			& 			& 			& 			& 			& 			\\
			31			&  119843548& 			& 			& 			& 			& 			& 			\\
			32			&  115345475& 			& 			& 			& 			& 			& 			\\
			33			&  81922759	& 			& 			& 			& 			& 			& 			\\
			34			&  44228481	& 			& 			& 			& 			& 			& 			\\
			35			&  18667991	& 			& 			& 			& 			& 			& 			\\
			36			&  6345554	& 			& 			& 			& 			& 			& 			\\
			37			&  1795212	& 			& 			& 			& 			& 			& 			\\
			38			&  437931	& 			& 			& 			& 			& 			& 			\\
			39			&  95241	& 			& 			& 			& 			& 			& 			\\
			40			&  18959	& 			& 			& 			& 			& 			& 			\\
			41			&  3517		& 			& 			& 			& 			& 			& 			\\
			42			&  617		& 			& 			& 			& 			& 			& 			\\
			43			&  101		& 			& 			& 			& 			& 			& 			\\
			44			&  16		& 			& 			& 			& 			& 			& 			\\
			45			&  2		& 			& 			& 			& 			& 			& 			\\
			\hline
		\end{tabular}
	}
	\caption{Some properties of the graphs in $\mL(12; 2; \leq 4)$}
	\label{table: mL_(+K_3)(12; 2; leq 4) properties}
\end{table}

\clearpage

\begin{table}[!h]
	\centering
	\resizebox{0.7\textwidth}{!}
	{
		\begin{tabular}{ | l r | l r | l r | l r | l r | l r | }
			\hline
			\multicolumn{2}{|c|}{\hbox to 3.5cm{$|\E(G)|$ \hfill $\#$}}&\multicolumn{2}{|c|}{\hbox to 3.5cm{$\delta(G)$ \hfill $\#$}}&\multicolumn{2}{|c|}{\hbox to 3.5cm{$\Delta(G)$ \hfill $\#$}}\\
			\hline
			48			& 1 		& 3			& 2782333	& 7			& 426		\\
			49			& 41 		& 4			& 248294425	& 8			& 269602932	\\
			50			& 1263 		& 5			& 1961917314& 9			& 3080309372\\
			51			& 24897 	& 6			& 1627736506& 10		& 535664232	\\
			52			& 340818 	& 7			& 51394620	& 11		& 6544240	\\
			53			& 3215961 	& 8			& 1676		& 12		& 5672		\\
			54			& 20943254 	& 			& 			& 			& 			\\
			55			& 94567255 	& 			& 			& 			& 			\\
			56			& 295234663 & 			& 			& 			& 			\\
			57			& 632937375 & 			& 			& 			& 			\\
			58			& 926347803 & 			& 			& 			& 			\\
			59			& 921306723 & 			& 			& 			& 			\\
			60			& 619034510 & 			& 			& 			& 			\\
			61			& 278204812 & 			& 			& 			& 			\\
			62			& 82280578 	& 			& 			& 			& 			\\
			63			& 15662269 	& 			& 			& 			& 			\\
			64			& 1876177 	& 			& 			& 			& 			\\
			65			& 141052 	& 			& 			& 			& 			\\
			66			& 7088 		& 			& 			& 			& 			\\
			67			& 314 		& 			& 			& 			& 			\\
			68			& 18		& 			& 			& 			& 			\\
			69			& 2 		& 			& 			& 			& 			\\
			\hline
		\end{tabular}
	}
	\caption{Some properties of the graphs in $\mL_{+K_3}(16; 1; 4)$}
	\label{table: mL_(+K_3)(16; 1; 4) properties}
\vspace{4em}
	\centering
	\resizebox{0.7\textwidth}{!}
	{
		\begin{tabular}{ | l r | l r | l r | l r | l r | l r | }
			\hline
			\multicolumn{2}{|c|}{\hbox to 3.5cm{$|\E(G)|$ \hfill $\#$}}&\multicolumn{2}{|c|}{\hbox to 3.5cm{$\delta(G)$ \hfill $\#$}}&\multicolumn{2}{|c|}{\hbox to 3.5cm{$\Delta(G)$ \hfill $\#$}}\\
			\hline
			86 & 317	& 8			& 19599716	& 9			& 35		\\
			87 & 8539	& 9			& 203852	& 10		& 6072772	\\
			88 & 94179	& 			& 			& 11		& 13316933	\\
			89 & 480821	& 			& 			& 12		& 411501	\\
			90 & 1574738	&		& 			& 13		& 2327		\\
			91 & 3492540	&		& 			& 			& 			\\
			92 & 5122647	&		& 			& 			& 			\\
			93 & 4864736	&		& 			& 			& 			\\
			94 & 2923601	&		& 			& 			& 			\\
			95 & 1026658	&		& 			& 			& 			\\
			96 & 194534	& 			& 			& 			& 			\\
			97 & 18960	& 			& 			& 			& 			\\
			98 & 1272	& 			& 			& 			& 			\\
			99 & 25		& 			& 			& 			& 			\\
			100& 1		& 			& 			& 			& 			\\
			\hline
		\end{tabular}
	}
	\caption{Some properties of the 20-vertex graphs obtained after the completion of step 4 of Algorithm \ref{algorithm: finding mL_(max)(n; 0; s)} $(n = 20, p = 0, s = 4)$}
	\label{table: 20-vertex 6-chromatic properties}
\end{table}


\clearpage


\clearpage

\begin{thebibliography}{10}
	
	\bibitem{AM18}
	V.~Angeltveit and B.~McKay
	\newblock {$R(5, 5) \leq 48$}.
	\newblock {\em Journal of Graph Theory}, 89(1):5-–13, 2018.
	
	\bibitem{Bik18}
	A.~Bikov.
	\newblock {Computation and bounding of Folkman numbers}.
	\newblock PhD Thesis. Sofia University "St. Kliment Ohridski", October 2018.
	\newblock Preprint: arxiv:1806.09601, June 2018
	
	\bibitem{Bik16}
	A.~Bikov.
	\newblock {Small minimal $(3, 3)$-Ramsey graphs}.
	\newblock {\em Ann. Univ. Sofia Fac. Math. Inform.}, 103:123-–147, 2016.
	\newblock Preprint: arxiv:1604.03716, April 2016.
	
	\bibitem{BN17}
	A.~Bikov and N.~Nenov.
	\newblock {The edge Folkman number $F_e(3, 3; 4)$ is greater than 19}.
	\newblock {Geombinatorics}, 27(1):5–-14, 2017.
	\newblock  Preprint: arxiv:1609.03468, September 2016.
	
	\bibitem{DR08}
	A.~Dudek and V.~R\"odl.
	\newblock {On the {F}olkman Number $f(2, 3, 4)$}.
	\newblock {\em Experimental Mathematics}, 17:63--67, 2008.
	
	\bibitem{Erd75}
	P.~Erd\H{o}s.
	\newblock {Problems and results on finite and infinite graphs}.
	\newblock {Recent Advances in Graph Theory \em Proc. Second Czechoslovak Sympos.}, Prague, 1974, 183--192, Academia, Prague, 1975.
	
	\bibitem{EH67}
	P.~Erd\H{o}s and A.~Hajnal.
	\newblock {Research problem 2-5}.
	\newblock {\em J. Combin. Theory}, 2:104, 1967.
	
	\bibitem{Fol70}
	J.~Folkman.
	\newblock {Graphs with monochromatic complete subgraph in every edge coloring}.
	\newblock {\em SIAM J. Appl. Math.}, 18:19--24, 1970.
	
	\bibitem{FR86}
	P.~Frankl and V.~R\"odl.
	\newblock {Large triangle-free subgraphs in graphs without $K_4$}.
	\newblock {\em Graphs and Combinatorics}, 2:135--144, 1986.
	
	\bibitem{Gra68}
	R.~L. Graham.
	\newblock {On edgewise 2-colored graphs with monochromatic triangles containing
		no complete hexagon}.
	\newblock {\em J. Combin. Theory}, 4:300, 1968.
	
	\bibitem{Gra12}
	R.~L. Graham.
	\newblock {Some Graph Theory Problems I Would Like to See Solved}.
	\newblock {\em SIAM My Favorite Graph Theory Conjectures}, Halifax, 2012.

	\bibitem{HI82}
	R.~Hill. and R.W.~Irwing.
	\newblock {On group partitions associated with lower bounds for symmetric Ramsey numbers.}.
	\newblock {\em European Journal of Combinatorics}, 3:35--50, 1982.

	\bibitem{Irw73}
	R.W.~Irwing.
	\newblock {On a bound of Graham and Spencer for a graph coloring constant.}.
	\newblock {\em Journal of Combinatorial Theory}, (B)15:200--203, 1973.

	\bibitem{KWR18}
	J. Kaufmann, H. Wickus, and S. Radziszowski.
	\newblock {On Some Edge Folkman Numbers Large and Small}.
	\newblock To appear in {\em Involve}.
	\newblock Preprint https://www.cs.rit.edu/~spr/PUBL/kwr18.pdf
	
	\bibitem{LRX14}
	A.~Lange, S.~Radziszowski, and X.~Xu.
	\newblock {Use of MAX-CUT for Ramsey Arrowing of Triangles}.
	\newblock {\em Journal of Comb. Math. and Comb. Comp.}, 88:61-71, 2014.
	
	\bibitem{Lin72}
	S.~Lin.
	\newblock {On Ramsey numbers and $K_r$-coloring of graphs}.
	\newblock {\em J. Combin. Theory Ser. B}, 12:82--92, 1972.
	
	\bibitem{Lu08}
	L.~Lu.
	\newblock {Explicit Construction of Small Folkman Graphs}.
	\newblock {\em SIAM J. on Discrete Math.}, 21:1053--1060, 2008.
	
	\bibitem{LRU01}
	T.~{\L}uczak, A.~Ruci{\'n}ski, and S.~Urba\'{n}ski
	\newblock {On minimal vertex Folkman graphs}.
	\newblock {\em Discrete Mathematics}, 236:245--262, 2001.
	
	\bibitem{McK_r}
	B.D.~McKay.
	\newblock \url{http://cs.anu.edu.au/~bdm/data/ramsey.html}.
	
	\bibitem{MP13}
	B.~D. McKay and A.~Piperino.
	\newblock Practical graph isomorphism, {II}.
	\newblock {\em J. Symbolic Computation}, 60:94--112, 2013.
	\newblock Preprint version at \href{http://arxiv.org/abs/1301.1493}{arxiv.org}.
	
	\bibitem{HN79}
	N.~Hadziivanov and N.~Nenov.
	\newblock {On the Graham-Spencer number (in Russian)}.
	\newblock {\em C. R. Acad. Bulg. Sci.}, 32:155--158, 1979.
	
	\bibitem{Nen81}
	N.~Nenov.
	\newblock {An example of a 15-vertex Ramsey (3, 3)-graph with clique number 4.
		(in Russian)}.
	\newblock {\em C. R. Acad. Bulg. Sci.}, 34:1487--1489, 1981.
	
	\bibitem{Nen83}
	N.~Nenov.
	\newblock {On the {Z}ykov numbers and some its applications to {R}amsey theory.
		(in Russian)}.
	\newblock {\em Serdica Bulg. math. publ.}, 9:161--167, 1983.
	
	\bibitem{Nen84}
	N.~Nenov.
	\newblock {The chromatic number of any 10-vertex graph without 4-cliques is at most 4.
		(in Russian)}.
	\newblock {\em C. R. Acad. Bulg. Sci.}, 37:301--304, 1984.
	
	\bibitem{PRU99}
	K.~Piwakowski, S.~Radziszowski, and S.~Urba\'{n}ski.
	\newblock {Computation of the Folkman number $F_e(3, 3; 5)$}.
	\newblock {\em J. Graph Theory}, 32:41--49, 1999.
	
	\bibitem{Rad14}
	S.~Radziszowski.
	\newblock {Small {R}amsey numbers}.
	\newblock {\em The Electronic Journal of Combinatorics}, Dynamic Survey
	revision 14, January 12 2014.
	
	\bibitem{RX07}
	S.~Radziszowski and X.~Xu.
	\newblock {On the Most Wanted Folkman Graph}.
	\newblock {\em Geombinatorics}, XVI(4):367--381, 2007.
	
	\bibitem{Soi08}
	A.~Soifer.
	\newblock {\em {The Mathematical Coloring Book}}.
	\newblock Springer, 2008.
	
	\bibitem{Spe88}
	J.~Spencer.
	\newblock {\em {Three hundred million points suffice}}.
	\newblock {\em J. Combin. Theory Ser. A}, 49:210--217, 1998.
	\newblock Also see erratum by M.~Hovey in 50:323.
	
	\bibitem{zchaff}
	{\em zchaff}, SAT Research Group, Princeton University.
	\newblock \url{https://www.princeton.edu/~chaff/zchaff.html}.
	
\end{thebibliography}
\end{document}